\newcommand{\bbC}{\mathbb{C}}

\newcommand{\bbQ}{\mathbb{Q}}

\newcommand{\bbZ}{\mathbb{Z}}

\newcommand{\Aut}{\textup{Aut}}
\newcommand{\Tr}{\textup{Tr}}

\newcommand{\Pic}{\textup{Pic}}

\documentclass[11pt]{amsart}
\usepackage{amscd, amsmath, amsthm, amssymb}
\newtheorem{theorem}{Theorem}[section]
\newtheorem{lemma}[theorem]{Lemma}
\newtheorem{proposition}[theorem]{Proposition}

\theoremstyle{definition}     

\newtheorem{remark}[theorem]{Remark}
\numberwithin{equation}{section}

\begin{document}

\title[K3 surfaces with an automorphism of order 66]{K3 surfaces with an automorphism of order 66, the maximum possible}

\author[J. Keum]{JongHae Keum}
\address{School of Mathematics, Korea Institute for Advanced Study, Seoul 130-722, Korea } \email{jhkeum@kias.re.kr}
\thanks{Research supported by National Research Foundation of Korea (NRF-2007-0093858).}


\date{Feb 2014}
\begin{abstract} In each characteristic $p\neq 2, 3$, it was shown in a previous work that the order of an automorphism of a K3 surface is bounded by 66, if finite. Here, it is shown that in each characteristic $p\neq 2, 3$ a K3 surface with a cyclic action of order $66$ is unique up to isomorphism. The equation of the unique surface is given explicitly in the tame case ($p\nmid 66$) and in the wild case ($p=11$).
\end{abstract}

\maketitle

 An automorphism of finite order
is called \emph{tame} if its order is prime to the characteristic, and \emph{wild} otherwise. Let $X$ be a K3 surface over an algebraically closed field $k$ of
characteristic $p\ge 0$. An automorphism $g$ of $X$ is called
\emph{symplectic} if it preserves a non-zero regular 2-form $\omega_X$, and
 \emph{purely non-symplectic} if no power of $g$ is symplectic except the identity. An automorphism of order a power of $p$ in  characteristic $p>0$ is symplectic, as there is no $p$-th root of unity.

In characteristic 0 or $p\neq 2, 3, 11$, Dolgachev (as recorded in Nikulin's paper \cite{Nik})  gave the first example of a K3 surface with an
automorphism of order 66. The K3 surface is the minimal model of the minimal resolution of the weighted hypersurface in ${\bf P}(33,22,6,1)$ defined by
\begin{equation}\label{Dolgaform}
x^2+y^3+z^{11}+w^{66}= 0
 \end{equation}
which has 3 singular points and whose minimal resolution has $K^2=-3$ (see \cite{HS}, p. 847). The affine model $x^2+y^3+z^{11}+1= 0$ is birational to $y^2+x^3+1-s^{11} = 0$, hence to the elliptic K3 surface in ${\bf P}(6,4,1,1)$
\begin{equation}\label{formula}
X_{66}: y^2+x^3+t_1^{12}-t_0^{11}t_1 = 0,
 \end{equation}
 as was later described by  Kond\=o \cite{Ko}. The surface $X_{66}$ has the automorphism
\begin{equation}\label{form2}
 g_{66}(x,y,t)=(\zeta_{66}^2x,\,\zeta_{66}^3y,\,\zeta_{66}^6t)
 \end{equation}
of order 66  where $t=t_1/t_0$ and $\zeta_{66}$ is a primitive 66th  root of unity.

In each characteristic $p\neq 2, 3$, it was shown in \cite{K} that the order of any automorphism of a K3 surface is bounded by 66, if finite. In this paper we characterize  K3 surfaces admitting a cyclic action of order $66$. 

For an automorphism $g$, tame or wild, of a K3 surface $X$, we
write
$${\rm ord}(g)=m.n$$ if $g$ is of order $mn$ and the homomorphism $\langle g\rangle\to {\rm GL}(H^0(X,
\Omega^2_X))$ has kernel of order $m$. 
A tame automorphism $g$ of order 66 of a K3 surface is purely non-symplectic  by \cite{K} (Lemma 4.2 and 4.4), i.e., $${\rm ord}(g)=1.66.$$

\begin{theorem} \label{main1} Let $k$ be the field  $\mathbb{C}$ of complex numbers or an algebraically closed
field of characteristic $p\neq 2, 3, 11$. Let $X$ be a K3 surface
defined over $k$ with an automorphism $g$ of order $66$. Then
$$(X, \langle g\rangle)\cong (X_{66}, \langle g_{66}\rangle),$$ i.e. there is an isomorphism $f:X\to X_{66}$ such that $f \langle g\rangle f^{-1}=\langle g_{66}\rangle$.
\end{theorem}

Over $k=\bbC$, Theorem \ref{main1} was proved by Kond\=o \cite{Ko} under the assumption that $g$ acts trivially on the Picard group of $X$, then  by Machida and Oguiso \cite{MO} under the assumption that $g$ is purely non-symplectic. Our proof is characteristic free and does not use the tools in the complex case such as transcendental lattice and the holomorphic Lefschetz formula.

The surface  $X_{66}$ is a weighted Delsarte surface. Using the algorithm  for determining the supersingularity (and the
Artin invariant) of such a surface whose minimal
resolution is a K3 surface (\cite{Shioda2}, \cite{Goto}), one can show that in characteristic $p\equiv -1$ (mod 66) the
surface $X_{66}$ is a supersingular K3 surface with Artin
invariant 1. Over $k=\bbC$, the Picard group of $X_{66}$ is a unimodular hyperbolic lattice of rank 2.

In characteristic $p=11$, there is an example of a K3 surface with a wild
automorphism of order 66 (\cite{DK3}, \cite{K}):
\begin{equation}\label{formula-11}
Y_{66}: y^2+x^3+t^{11}-t = 0,
 \end{equation}
\begin{equation}\label{form2-11}
 h_{66}(x,y,t)=(\zeta_{6}^2x,\,\zeta_{6}^3y,\,t+1)
 \end{equation}
 where $\zeta_{6}\in k$ is a primitive 6th  root of unity.
The surface is a supersingular K3 surface with Artin invariant 1 in characteristic $p=11$ (mod 12).

\begin{theorem} \label{main2} Let $k$ be an algebraically closed
field of characteristic $p=11$. Let $X$ be a K3 surface
defined over $k$ with an automorphism $g$ of order $66$. Then
 $${\rm  ord}(g)= 11.6$$ and
$$(X, \langle g\rangle)\cong(Y_{66}, \langle h_{66}\rangle),$$ i.e. there is an isomorphism $f:X\to Y_{66}$ such that $f \langle g\rangle f^{-1}=\langle h_{66}\rangle$.
\end{theorem}

\begin{remark} In characteristic $p=2$, $3$, there is an example of a K3 surface with an automorphism of order 66, as was noticed by Matthias Sch\"utt:

\medskip
$p=2$
\begin{equation}\label{formula-2}
X: y^2-y=x^3+t^{11},
 \end{equation}
\begin{equation}\label{form2-2}
 g(x,y,t)=(\zeta_{33}^{11}x,\, y+1,\, \zeta_{33}^3t)
 \end{equation}
 where $\zeta_{33}\in k$ is a primitive 33rd  root of unity. The surface has only one singular fibre, type $II$ at $t=\infty$. All smooth fibres have $j$-invariant 0.

 \medskip
$p=3$
\begin{equation}\label{formula-3}
X: y^2=x^3-x+t^{11},
 \end{equation}
\begin{equation}\label{form2-3}
g(x,y, t)=(x+1,\, \zeta_{22}^{11}y,\, \zeta_{22}^2t)
 \end{equation}
 where $\zeta_{22}\in k$ is a primitive 22nd root of unity. The surface has only one singular fibre, type $II$ at $t=\infty$. All smooth fibres have $j$-invariant 0.
 
 In characteristic $2$ and 3, it seems that 66 is the maximum finite order  and is realized only by the above surface up to isomorphism.
\end{remark}

\bigskip

{\bf Notation}

\bigskip
For an automorphism $g$ of a K3 surface $X$, we use the following notation:
\begin{itemize}
\item $X^g={\rm Fix}(g)$ : the fixed locus of $g$
\medskip
\item $e(g):=e({\rm Fix}(g))$, the Euler characteristic of ${\rm Fix}(g)$;
\medskip
\item $\Tr(g^*|H^*(X)):=\sum_{j=0}^{2\dim X} (-1)^j\Tr (g^*|H^j_{\rm et}(X,{\bbQ}_l))$.
\medskip
\item $[g^*]=[\lambda_1, \ldots, \lambda_{22}]$ : the eigenvalues of $g^*|H^2_{\rm et}(X,{\bbQ}_l)$
\medskip
\item $\zeta_a$ : a primitive $a$-th  root of unity in $\overline{\bbQ_l}$
\medskip
\item $\zeta_a:\phi(a)$ : all primitive $a$-th  roots of unity in $\overline{\bbQ_l}$ where $\phi$ is the Euler function and $\phi(a)$ the number of conjugates of $\zeta_a$
\medskip
\item $[\lambda.r]\subset [g^*]$ : $\lambda$ repeats $r$ times in
$[g^*]$.
\medskip
\item $[(\zeta_a:\phi(a)).r]\subset [g^*]$ : the list $\zeta_a:\phi(a)$ repeats $r$ times in
$[g^*]$.
\end{itemize}

\section{Preliminaries}

The following basic
results can be found in the previous paper \cite{K}.

\begin{proposition}\label{integral}$($Proposition 2.1 \cite{K}$)$ Let $g$ be an automorphism of a projective variety $X$ over an algebraically closed field $k$ of
characteristic $p> 0$. Let $l$ be a prime $\neq p$. Then the following hold true.
\begin{enumerate}
\item $($3.7.3 \cite{Illusie}$)$ The characteristic polynomial of
$g^*|H_{\rm et}^j(X,\bbQ_l)$ has integer coefficients for each
$j$. The characteristic polynomial does not depend on the choice of  cohomology, $l$-adic or crystalline. In particular, if a primitive
$m$-th root of unity appears with multiplicity $r$ as an
eigenvalue of $g^*|H_{\rm et}^j(X,\bbQ_l)$, then so does each of
its conjugates.
\item  If $g$ is of finite order, then $g$ has an
invariant ample divisor, and
$g^*|H_{\rm et}^2(X,\bbQ_l)$ has $1$ as an eigenvalue.
\item  If $X$ is a K3 surface, $g$ is tame and $g^*|H^0(X,\Omega_X^2)$ has
$\zeta_n\in k$ as an eigenvalue, then $g^*|H_{\rm et}^2(X,\bbQ_l)$
has $\zeta_n\in \overline{\bbQ_l}$ as an eigenvalue.
\end{enumerate}
\end{proposition}

\begin{proposition}\label{trace}$($Topological Lefschetz formula, cf. \cite{DL} Theorem 3.2$)$   Let $X$ be a smooth projective variety over  an algebraically closed field $k$ of
characteristic $p> 0$ and let $g$ be a tame automorphism of $X$. Then $X^g={\rm Fix}(g)$ is smooth and
 $$e(g):=e(X^g)=\Tr(g^*|
H^*(X)).$$
\end{proposition}



\begin{lemma}\label{nsym2}$($Lemma 1.6 \cite{K60}$)$   
Let $X$ be a K3 surface in characteristic $p\neq 2$, admitting an automorphism $h$ of order $2$ with $\dim
H^2_{\rm et}(X,{\bbQ}_l)^h =2$. Then $h$ is non-symplectic and has
an $h$-invariant elliptic fibration $\psi:X\to {\bf P}^1$,  $$X/\langle h\rangle\cong {\bf F}_e$$ a rational ruled surface, and $X^h$ is either a curve of genus
$9$ which is a $4$-section of $\psi$ or the union of a section and a curve of genus $10$ which
is a $3$-section.
In the first case $e=0, 1$ or $2$, and in the second $e=4$. Each singular fibre of $\psi$
is of type $I_1$ $($nodal$)$, $I_2$, $II$ $($cuspidal$)$ or $III$, and is intersected by $X^h$ at the node and two smooth points if  of type $I_1$, at the two singular points if of type $I_2$, at the cusp with multiplicity $3$ and a smooth point if of type $II$, at the singular point tangentially to both components if of type $III$. If $X^h$ contains a section, then each singular fibre
is of type $I_1$ or $II$.
\end{lemma}

\begin{remark} If $e\neq 0$, the $h$-invariant elliptic fibration $\psi$ is the pull-back of the unique ruling of ${\bf F}_e$. If $e=0$, either ruling of ${\bf F}_0$ lifts to an $h$-invariant elliptic fibration.
\end{remark}

\begin{lemma}\label{fix}$($Lemma 2.10 \cite{K}$)$ Let $S$ be a set and ${\rm Aut}(S)$ be the group of bijections of $S$. For any $g\in {\rm Aut}(S)$ and positive integers $a$ and $b$,
\begin{enumerate}
\item ${\rm Fix}(g)\subset {\rm Fix}(g^a)$;
\item ${\rm Fix}(g^a)\cap {\rm Fix}(g^b)={\rm Fix}(g^d)$ where $d=\gcd (a, b)$;
\item ${\rm Fix}(g)= {\rm Fix}(g^a)$ if ${\rm ord}(g)$ is finite and  prime to $a$.
\end{enumerate}
\end{lemma}

\begin{lemma}\label{sum}$($Lemma 2.11 \cite{K}$)$ Let $R(n)$ be the sum of all primitive $n$-th root of unity in $\overline{\bbQ}$ or in
$\overline{\bbQ_l}$ where $(l,n)=1$. Then
$$R(n)=\left\{\begin{array}{ccl} 0&{\rm if}& n\,{\rm has\,\, a\,\, square\,\, factor},\\
(-1)^t&{\rm if}& n\,{\rm is\,\, a\,\, product\,\, of}\,\,t\,\,{\rm distinct\,\, primes}.\\
\end{array} \right.$$
\end{lemma}

\section{Invariant elliptic fibration}

The following two lemmas will play a key role in our proof.

\begin{lemma}\label{66} Let $g$ be an automorphism of order $66$ of a K3
surface $X$ in characteristic $p\neq 2$, $3$, $11$. If the eigenvalues of $g^*$ on the second cohomology is given by $$[g^{*}]=[1,\,\zeta_{66}:20,\, 1],$$ then
\begin{enumerate}
\item there is a $g^{}$-invariant elliptic
fibration $\psi:X\to {\bf P}^1$ with $12$ cuspidal fibres, say $F_0$, $F_{t_1}, \ldots, F_{t_{11}}$;
\item  ${\rm Fix}(g^{33})$
consists of a section $R$ of $\psi$ and a curve $C_{10}$ of genus
$10$ which is a $3$-section passing through each cusp with multiplicity $3$;
\item the action of $g$ on the base ${\bf P}^1$ is of order $11$, fixes $2$ points, say $\infty$ and $0$, and makes the $11$ points $t_1, \ldots, t_{11}$ form a single orbit, where $F_{\infty}$ is a smooth fibre;
\item $${\rm Fix}(g^{11})=R\cup\,\{{\rm the\,
cusps\, of\, the\, 12\, cuspidal\, fibres}\};$$
\item ${\rm Fix}(g)$  consists of the  
$3$ points,  $$R\cap F_{\infty},\,\,R\cap F_{0},\,\,C_{10}\cap F_{0}.$$
\end{enumerate}
\end{lemma}

\begin{proof}
Note that
$[g^{33*}]=[1,\,-1.20,\, 1]$. 
Thus, we can apply Lemma \ref{nsym2} to $h=g^{33}$. We compute $e(g)=3$ and
$$[g^{11*}]=[1,\,(\zeta_{6}:2).10,\, 1],\,\,\,e(g^{11})=14.$$ 
Note that $${\rm Fix}(g^{d})\subset {\rm Fix}(g^{33})$$ for any $d$ dividing 33.
If ${\rm Fix}(g^{33})$ is a curve $C_{9}$ of genus $9$, then $g^{11}$ acts on $C_9$ with 14 fixed points,
too many for an order 3 automorphism on a curve of genus 9.   Thus $$X/\langle g^{33} \rangle\cong {\bf F}_4,$$ there is a
$g^{33}$-invariant elliptic fibration $$\psi:X\to {\bf P}^1$$ and ${\rm Fix}(g^{33})$ consists of a section $R$ of $\psi$ and
a curve $C_{10}$ of genus $10$ which is a $3$-section.   The automorphism $\bar{g}$ of ${\bf F}_4$ induced by $g$ preserves
the unique ruling, so $g$ preserves the elliptic fibration. 
Since $\bar{g}^{33}$ acts trivially on ${\bf F}_4$,
$g^{33}$ acts trivially on the base ${\bf P}^1$, and hence the orbit of a fibre under the action of  $g|{\bf P}^1$ has length 1, 3, 11 or 33 . 
By Lemma \ref{nsym2} a fibre of $\psi$
is of type $I_0$ (smooth), $I_1$, $I_2$, $II$ or $III$. Claim that $\psi$ has no fibre of type $I_2$ or $III$. If a fibre $F$ is of type $III$, then its orbit under $g|{\bf P}^1$ has length 1 or 3, then $g^3$ preserves $F$, hence $g^6$ preserves both components of $F$ and, together with an invariant ample class, preserves 3 linearly independent classes, hence $[g^{6*}]\supset [1,\,1,\,1]$, impossible. If a fibre $F$ is of type $I_2$, then its orbit under $g|{\bf P}^1$ has length 1, 3 or 11, then $g^6$ or $g^{22}$ would have 3 linearly independent invariant classes, again impossible. Next, claim that 
there is an orbit of singular fibres of length 11. Otherwise, all orbits of singular fibres would have length 1 or 3, then $g^3$ would preserve all fibres, hence fix the curve $R$ and induces on a general smooth fibre an automorphism of order 22. But in any characteristic no elliptic curve admits an automorphism of such high order that fixes a point. 
If there are two orbits of length 11 of singular fibres of type $I_1$, then $g^{11}$ would preserve all fibres, hence fix $R$ and the singular points of singular fibres, then $e(g^{11})>14$.
Thus there is one orbit of length 11 of singular fibres of type $II$. If $g$ preserves two fibres of type $I_1$, then the same argument as above would yield $e(g^{11})>14$. Thus $g$ preserves one fibre of type $II$ and a smooth fibre.
 This proves (1), (2) and (3).

The statement (4) follows from (3) and the fact that ${\rm Fix}(g^{11})$ has Euler number 14 and is contained in $R\cup C_{10}$. 

To see (5), take $R$ as the 0-section of $\psi$. Then on each smooth fibre $F$, $g^{11}$ induces an order 6 automorphism, fixing the point $F\cap R$ and rotating the three 2-torsions $C_{10}\cap F$.
\end{proof}

\begin{lemma}\label{66'} Let $g$ be an automorphism of order $66$ of a K3
surface $X$ in characteristic $p=11$. If $$[g^{*}]=[1,\,\zeta_{66}:20,\, 1],$$ then
\begin{enumerate}
\item there is a $g^{}$-invariant elliptic
fibration $\psi:X\to {\bf P}^1$ with $12$ cuspidal fibres, say $F_{\infty}$, $F_{t_1}, \ldots, F_{t_{11}}$;
\item  ${\rm Fix}(g^{33})$
consists of a section $R$ of $\psi$ and a curve $C_{10}$ of genus
$10$ which is a $3$-section passing through each cusp with multiplicity $3$;
\item the action of $g$ on the base ${\bf P}^1$ is of order $11$, fixes $1$ point, say $\infty$, and makes the $11$ points $t_1, \ldots, t_{11}$  form a single orbit;
\item $${\rm Fix}(g^{11})=R\cup\,\{{\rm the\,
cusps\, of\, the\, 12\, cuspidal\, fibres}\};$$
\item ${\rm Fix}(g)$  consists of the  
$2$ points,  $$R\cap F_{\infty},\,\,C_{10}\cap F_{\infty}.$$
\end{enumerate}
\end{lemma}

\begin{proof} The proof of Lemma \ref{66}, with a modification, will work here. 
Note first that the automorphisms $g^{33}$,  $g^{11}$ are tame in  characteristic $p=11$, so the Lefschetz fixed point formula holds for them and the argument using their fixed loci is valid.

We compute
$[g^{33*}]=[1,\,-1.20,\, 1]$ and apply Lemma \ref{nsym2} to $h=g^{33}$. We also compute 
$$[g^{11*}]=[1,\,(\zeta_{6}:2).10,\, 1],\,\,\,e(g^{11})=14.$$ 
Note that $${\rm Fix}(g^{d})\subset {\rm Fix}(g^{33})$$ for any $d$ dividing 33.
If ${\rm Fix}(g^{33})$ is a curve $C_{9}$ of genus $9$, then $g^{11}$ acts on $C_9$ with 14 fixed points,
too many for an order 3 automorphism on a curve of genus 9.   Thus $$X/\langle g^{33} \rangle\cong {\bf F}_4,$$ there is a
$g^{33}$-invariant elliptic fibration $$\psi:X\to {\bf P}^1$$ and ${\rm Fix}(g^{33})$ consists of a section $R$ of $\psi$ and
a curve $C_{10}$ of genus $10$ which is a $3$-section.   The automorphism $\bar{g}$ of ${\bf F}_4$ induced by $g$ preserves
the unique ruling, so $g$ preserves the elliptic fibration. Note that
$g^{33}$ acts trivially on the base ${\bf P}^1$. 
By Lemma \ref{nsym2} a fibre of $\psi$
is of type $I_0$ (smooth), $I_1$, $I_2$, $II$ or $III$. By the same argument as in Lemma \ref{66}, $\psi$ has no fibre of type $I_2$ or $III$ and 
there is an orbit of singular fibres of length 11. 
If there are two orbits of length 11 of singular fibres of type $I_1$, then $g^{11}$ would preserve all fibres, hence fix $R$ and the singular points of singular fibres, then $e(g^{11})>14$.
Thus there is one orbit of length 11 of singular fibres of type $II$. If $g$ preserves two fibres of type $I_1$, then the same argument as above would yield $e(g^{11})>14$. Thus $g$ preserves one fibre of type $II$. Since  $g|{\bf P}^1$ is of order 11 and an wild automorphism on ${\bf P}^1$ fixes only one point, we see that $g$ preserves no other fibre.
 This proves (1), (2) and (3).

The statement (4) follows from (3). Since $g^{11}$ is tame, ${\rm Fix}(g^{11})$ has Euler number 14 and is contained in $R\cup C_{10}$.

To see (5), note that ${\rm Fix}(g)$ is contained not only in the cuspidal fibre $F_{\infty}$ but also in ${\rm Fix}(g^{33})=C_{10}\cup R$.
\end{proof}

\section{the Tame Case}

Throughout this section, we assume that the characteristic $p>0$,
$p\nmid 66$, and $g$ is an automorphism of order $66$ of a K3
surface.
By \cite{K} Lemma 4.2 and 4.4, $g$ is purely non-symplectic, i.e.
 ${\rm  ord}(g)= 1.66$.

\begin{lemma} The eigenvalues of $g^*$ on the second cohomology is given by
\label{g*} $$[g^{*}]=[1,\,\zeta_{66}:20,\,1].$$
\end{lemma}

\begin{proof}
By Proposition
\ref{integral} the action of $g^*$ on $H_{\rm et}^2(X,\bbQ_l)$ has
$\zeta_{66}\in \overline{\bbQ_l}$ as an eigenvalue. Thus $[\zeta_{66}:20]\subset[g^*]$. 
Suppose that 
$$[g^{*}]=[1,\,\zeta_{66}:20,\,-1].$$ Then $$[g^{33*}]=[1,\,-1:20,\,-1], \quad e(g^{33})=-18.$$ Since the $g^{33*}$-invariant subspace of $H_{\rm
et}^2(X,\bbQ_l)$  has dimension 1, we see that 
$${\rm Fix}(g^{33})=C_{10},$$
 a smooth curve of genus 10, and the quotient surface
$$X/\langle g^{33}\rangle\cong {\bf P}^2.$$ 
The image ${C_{10}'}\subset  {\bf P}^2$ is a smooth sextic curve. Since $e(g^{11})=12$ and 
$${\rm Fix}(g^{11})\subset {\rm Fix}(g^{33})=C_{10},$$we see that  ${\rm Fix}(g^{11})$ consists of 12 points. On the other hand, $g^{22}$ has
$$[g^{22*}]=[1,\,(\zeta_{3}:2).10,\, 1],\quad e(g^{22})=-6.$$
Since the $g^{*22}$-invariant subspace of $H_{\rm
et}^2(X,\bbQ_l)$  has dimension 2, ${\rm Fix}(g^{22})$ consists of a curve $C$ of genus $>1$, at most one ${\bf P}^1$ and some isolated  points. If ${\rm Fix}(g^{22})$ contains  a ${\bf P}^1$, then the action of  $g$ on ${\rm Fix}(g^{22})$   preserves $C$ and the ${\bf P}^1$, so the $g^{*}$-invariant subspace of $H_{\rm
et}^2(X,\bbQ_l)$  has dimension at least 2, a contradiction. Here we use the fact that
the Chern class map 
$$c_1: \Pic(X)\to H_{\rm crys}^2(X/W)$$
is injective and the fact that the characteristic polynomial of $g^*$ does not depends on the choice of cohomology.
Thus  ${\rm Fix}(g^{22})$ contains no ${\bf P}^1$ and 
$${\rm Fix}(g^{22})= C_{k+4}\cup\{\, 2k\,\, {\rm points}\,\}$$
for a smooth curve $ C_{k+4}$ of genus $k+4$. Note that
$$ C_{k+4}\cap C_{10}\subset {\rm Fix}(g^{22})\cap {\rm Fix}(g^{33})={\rm Fix}(g^{11}),$$ thus the intersection number $$C_{k+4}.C_{10}\le 12.$$ Then the
Hodge Index Theorem gives
$$(C_{k+4}^2)( C_{10}^2)=18(2k+6)\le (C_{k+4}. C_{10})^2\le 12^2,$$
thus $k\le 1$. 

\medskip
Suppose that $k=0$ and ${\rm Fix}(g^{22})= C_{4}$. Since $C_{4}^2/ C_{10}^2$ is not a square of a rational number, the two curves
$C_4$ and $C_{10}$ are linearly independent in Pic$(X)\otimes{\mathbb Q}$,  giving two linearly independent $g^{*}$-invariant vectors of $H_{\rm
et}^2(X,\bbQ_l)$, a contradiction.

\medskip
Suppose that $k=1$ and ${\rm Fix}(g^{22})= C_{5}\cup\{\, 2\,\, {\rm points}\, \}$.
Since $(C_{5}^2)( C_{10}^2)=144$, we have the equality in the
Hodge Index Theorem and $$C_{5}.C_{10}= 12.$$ Since $g^{33}|C_5=g^{11}|C_5$, the action of $g^{33}$ on $C_5$ has 12 fixed points, hence the image $$C_5'\subset X/\langle g^{33}\rangle\cong {\bf P}^2$$ has genus 0. Since $C_{5}'.C_{10}'= 12$, $C_5'$ must be a smooth conic. Consider the automorphism  $\bar{g}^{11}$ of $X/\langle g^{33}\rangle$ induced by $g^{11}$. It has order 3 and its fixed locus $ {\rm Fix}(\bar{g}^{11})\subset {\bf P}^2$ is the image of the locus 
$$ {\rm Fix}(g^{11})\cup {\rm Fix}(g^{22})= {\rm Fix}(g^{22}),$$
hence $ {\rm Fix}(\bar{g}^{11})$ consists of the conic $C_5'$ and the point which is the image of the two points in ${\rm Fix}(g^{22})$ . But the fixed locus of any order 3 automorphism of ${\bf P}^2$ is either 3 isolated points or the union of a point and a line, a contradiction.
\end{proof}

\bigskip
\noindent {\bf Proof of Theorem \ref{main1}.} 

 By Lemma \ref{g*}, $[g^{*}]=[1,\,\zeta_{66}:20,\,1]$.  We can apply Lemma \ref{66}, and will use the elliptic structure and the notation.
Let
$$y^2+x^3+A(t_0,t_1)x+B(t_0,t_1) = 0$$ be the Weierstrass equation of  the $g$-invariant elliptic
pencil, where $A$ (resp. $B$) is a binary form of degree  $8$
(resp. $12$). By Lemma \ref{66}, $g$ leaves invariant the
section $R$ and the action of $g$ on the base of the fibration
$\psi:X\to {\bf P}^1$ is of order 11. After a linear change of the
coordinates $(t_0,t_1)$ we may assume that $g$ acts on the base by
$$g:(t_0,t_1)\mapsto (t_0,\zeta_{11}t_1)$$ for some primitive 11th root of unity $\zeta_{11}$.
We know that $g$ preserves one cuspidal fibre $F_0$
and makes the remaining 11 cuspidal fibres form one orbit. Thus
the discriminant polynomial \begin{equation}\label{Delta}
\Delta =
-4A^3-27B^2=ct_1^2(t_1^{11}-t_0^{11})^2
 \end{equation} for some constant
$c\in k$, as it must have one double root (corresponding to the
fibres $F_0$) and one orbit of double roots. From the equality (\ref{Delta}), it is easy to see that $A$ is not a non-zero constant. If deg$(A)>0$, then the zeros of $A$ correspond to either cuspidal fibres (which may contain a singular point of $X$, i.e. yield a reducible fibre) or
nonsingular fibres with ``complex multiplication'' of
order 6. This set has cardinality at most 8, but invariant with respect to the order 11
action of $g|{\bf P}^1$, impossible. Thus $A
= 0$. Then the above Weierstrass equation can be written in
the form
\begin{equation}\label{wform}
y^2+x^3+at_1(t_1^{11}-t_0^{11}) = 0
 \end{equation} for some constant $a$. A suitable
linear change of variables makes $a=1$ without changing the action
of $g$ on the base. Thus $$X \cong X_{66}$$ as an elliptic surface. Let $$t=t_1/t_0.$$ Choose a primitive 66th root of unity $\zeta_{66}$ such that
\begin{equation}\label{g*omega}
g^*\Big(\frac{dx\wedge dt}{y}\Big)=\zeta_{66}^5\frac{dx\wedge dt}{y},\,\,\,\,g^{11*}\Big(\frac{dx\wedge dt}{y}\Big)=\zeta_{66}^{55}\frac{dx\wedge dt}{y}.\end{equation}  Since $g^{11}$ is of order 6, acts trivially
on the base  and fixes the section $R$,  it is a complex multiplication of order 6 on a general fibre, so
$$g^{11}(x, y, t)=(\zeta_{66}^{22}x, -y, t).$$
Here, the other primitive 3rd root of unity $\zeta_{66}^{44}$ cannot appear as the coefficient of $x$ by (\ref{g*omega}).  We will analyse the local action of $g$
at the fixed point $(x, y, t)=(0,0,0)$, the cusp of
$F_0$. We  first determine the linear terms of $g$, then infer that the higher degree terms must vanish. Write the linear terms of $g$ as follows:
$$g(x, y, t)=(\zeta_{66}^ax, \zeta_{66}^by,
\zeta_{66}^ct).$$
 Since the Weierstarass equation (\ref{wform}) is invariant under $g$, we have the following system of congruence modulo 66:
$$3a\equiv 2b\equiv 12c\equiv c$$
$$11a\equiv 22$$
$$11b\equiv 33.$$
The solutions are
$$a\equiv 2+6a'$$
$$b\equiv 3+9a'\,\,(a'\,\,{\rm even})\,\,{\rm or}\,\,36+9a'\,\,(a'\,\,{\rm odd})$$
$$c\equiv 6+18a'$$ for some integer $a'$.
On the other hand, by (\ref{g*omega})
$$5\equiv a+c-b \,\,\,({\rm mod}\,\,66).$$
This congruence equation is satisfied by the solution $a\equiv 2$, $b\equiv 3$, $c\equiv 6$, but by no other solution among the above solutions.
This completes the proof of Theorem \ref{main1} in the tame case.

\section{the Complex Case}

We may assume that $X$ is projective, since a non-projective complex K3 surface cannot admit a non-symplectic
automorphism of finite order (see \cite{Ueno}, \cite{Nik}) and
its automorphisms of finite order are symplectic, hence of order
$\le 8$.  Now the same  proof goes, once $H^2_{\rm
et}(X,\bbQ_l)$ is replaced by $H^2(X,\bbZ)$.

\section{in characteristic $p=11$}

Throughout this section, we assume that the characteristic $p=11$ and $g$ is an automorphism of order $66$ of a K3 surface. 
By \cite{DK3} we know that ${\rm  ord}(g)= 11.6$.

\begin{lemma}\label{11.6}   ${\rm  ord}(g)= 11.6$.
\end{lemma}

\begin{proof} Any automorphism of order $p$ in characteristic $p$ is symplectic, as there is no $p$-th root of unity. Thus the symplectic order of $g$ must be a  multiple of 11. In characteristic 11 it is known \cite{DK3}  that 11 is the maximum possible among all orders of symplectic automorphisms of finite order.
\end{proof}

\begin{lemma}\label{g}  The eigenvalues of $g^*$ on the second cohomology is given by  $$[g^{*}]=[1,\,\zeta_{66}:20,\,1].$$
\end{lemma}

\begin{proof} In  characteristic $p=11$ it was proved in \cite{DK3} Proposition 4.2 that the representation on $H_{\rm et}^2(X,\bbQ_l)$ of a finite group of symplectic automorphisms is Mathieu. It follows that the order 11 automorphism $g^{6}$ has
$$[g^{6*}]=[1,\,(\zeta_{11}:10).2,\,1].$$ There is a $g$-invariant ample divisor class, so 1 appears in $[g^{*}]$. Since the representation of $\Aut(X)$ on $H_{\rm et}^2(X,\bbQ_l)$ is faithful (\cite{Ogus} Corollary 2.5, \cite{K} Theorem 1.4), $g^{*}|H_{\rm et}^2(X,\bbQ_l)$ has order 66 and we infer that $[g^{*}]$ is one of the following 3 cases:
$$[g^{*}]=[1,\,\zeta_{66}:20,\,\pm 1],\quad [1,\,\zeta_{33}:20,\,-1].$$
On the other hand, $g^{11}$ is tame and non-symplectic of order 6, hence $\zeta_6 \in [g^{11*}]$ by Proposition \ref{integral}. This excludes the last case.

\medskip
Suppose that $[g^{*}]=[1,\,\zeta_{66}:20,\,-1].$ This case can be ruled out by the  same proof as in Lemma \ref{g*}. Indeed, the automorphisms $g^{33}$, $g^{22}$, $g^{11}$ are tame in  characteristic $p=11$, so the Lefschetz fixed point formula holds for them and the argument using their fixed loci is valid.
\end{proof}

\bigskip
\noindent {\bf Proof of Theorem \ref{main2}.} 

The first statement follows from Lemma \ref{11.6}. It remains to prove the second.
 By Lemma \ref{g}, $[g^{*}]=[1,\,\zeta_{66}:20,\,1]$.  We can apply Lemma \ref{66'} and will use the elliptic structure and the notation.
Let
$$y^2+x^3+A(t_0,t_1)x+B(t_0,t_1) = 0$$ be the Weierstrass equation of  the $g$-invariant elliptic
pencil, where $A$ (resp. $B$) is a binary form of degree  $8$
(resp. $12$). By Lemma \ref{66'}, $g$ leaves invariant the
section $R$ and the action of $g$ on the base of the fibration
$\psi:X\to {\bf P}^1$ is of order 11.  Any wild automorphism of ${\bf P}^1$ is uni-potent, so after a linear change of the
coordinates $(t_0,t_1)$ we may assume that $g$ acts on the base by
$$g:(t_0,t_1)\mapsto (t_0,t_1+t_0).$$
Then $g$ preserves the cuspidal fibre $F_{\infty}$
and makes the remaining 11 cuspidal fibres form one orbit. Thus
the discriminant polynomial $$\Delta =
-4A^3-27B^2=ct_0^2(t_1^{11}-t_0^{10}t_1)^2$$ for some constant
$c\in k$, as it must have one double root (corresponding to the
fibres $F_{\infty}$) and one orbit of double roots. The zeros of $A$ correspond to either cuspidal fibres (which may contain a singular point of $X$) or
nonsingular fibres with ``complex multiplication'' of
order 6. Since this set is invariant with respect to the order 11
action of $g|{\bf P}^1$, we see that the only possibility is $A
= 0$. Then the above Weierstrass equation can be written in
the form
$$y^2+x^3+at_0(t_1^{11}-t_0^{10}t_1) = 0$$ for some constant $a$. A suitable
linear change of variables makes $a=1$ without changing the action
of $g$ on the base. Thus $$X \cong Y_{66}$$ as an elliptic surface. Let $$t=t_1/t_0.$$ 
Since $g$ has non-symplectic order 6, one can choose a primitive 6th root of unity $\zeta_{6}$ such that 
$$g^*\Big(\frac{dx\wedge dt}{y}\Big)=\zeta_{6}^{-1}\frac{dx\wedge dt}{y},\quad g^{11*}\Big(\frac{dx\wedge dt}{y}\Big)=\zeta_{6}\frac{dx\wedge dt}{y}.$$  Since $g^{11}$ is of order 6, acts trivially
on the base and fixes the section $R$, it is a complex multiplication of order 6 on a general fibre, so
$$g^{11}(x, y, t)=(\zeta_6^4x, \zeta_6^3y, t).$$ We know that $g(t)=t+1$,  so infer that
$$g(x, y, t)=(\zeta_{6}^2x, \zeta_{6}^3y, t+1).$$ 




\end{document}